\documentclass[11pt,reqno]{amsart}
\usepackage{enumerate, latexsym, amsmath, amsfonts, amssymb, amsthm, color}

\makeatletter
\renewcommand{\@seccntformat}[1]{{\csname the#1\endcsname}.\hspace{.5em}}
\makeatother

\newtheorem{thm}{Theorem}[section]

\newtheorem{lem}[thm]{Lemma}
\newtheorem{remark}[thm]{Remark}

\def\pmod #1{\ ({\rm{mod}}\ #1)}
\def\Z{\Bbb Z}
\def\N{\Bbb N}

\def\bg{\bigg}
\def\({\bg(}
\def\){\bg)}
\def\t{\text}
\def\f{\frac}

\def\bi{\binom}

\def\eq{\equiv}

\renewcommand{\qed}{\hfill$\Box$\medskip}

\numberwithin{equation}{section}

\begin{document}

\hbox{Preprint}
\medskip

\title
[{Proof of a conjectural supercongruence modulo $p^5$}]
{Proof of a conjectural supercongruence modulo $p^5$}

\author
[Guo-Shuai Mao and Zhi-Wei Sun] {Guo-Shuai Mao and Zhi-Wei Sun}

\address{(Guo-Shuai Mao) Department of Mathematics, Nanjing
University of Information Science and Technology, Nanjing 210044,  People's Republic of China}
\email{maogsmath@163.com}

\address{(Zhi-Wei Sun) Department of Mathematics, Nanjing
University, Nanjing 210093, People's Republic of China}
\email{zwsun@nju.edu.cn}


\begin{abstract}
In this paper we prove the supercongruence
$$\sum_{n=0}^{(p-1)/2}\frac{6n+1}{256^n}\binom{2n}n^3\equiv p(-1)^{(p-1)/2}+(-1)^{(p-1)/2}\frac{7}{24}p^4B_{p-3}\pmod{p^5}$$
for any prime $p>3$, which was conjectured by Sun in 2019.
\end{abstract}
\maketitle

\section{Introduction}

In 1997, L. van Hamme \cite{vhamme} proposed many conjectural $p$-adic supercongruences
motivated by corresponding Ramanujan-type series for $1/\pi$. For example, he conjectured the supercongruence
\begin{equation}
\label{p^4}\sum_{n=0}^{(p-1)/2}\frac{6n+1}{256^n}\binom{2n}n^3\equiv (-1)^{(p-1)/2}p\pmod{p^4}
\end{equation}
for any prime $p>3$, inspired by the Ramanujan series (cf. \cite{Ram})
$$\sum_{n=0}^\infty\frac{6n+1}{256^n}\binom{2n}n^3=\frac 4{\pi}.$$
The congruence $\eqref{p^4}$ was confirmed by L. Long \cite{long-2011-pjm} in 2011.

In 2011 Z.-W. Sun \cite{sun-scm-2011} formulated many conjectural supercongruences involving Bernoulli numbers or Euler numbers. Recall that the Bernoulli numbers $B_0,B_1,\ldots$ and the Euler numbers
$E_0,E_1,\ldots$ are defined by
$$\frac x{e^x-1}=\sum_{n=0}^\infty B_n\frac{x^n}{n!}\ (|x|<2\pi)
\ \text{and}\ \frac{2}{e^x+e^{-x}}=\sum_{n=0}^\infty E_n\frac{x^n}{n!}\ \left(|x|<\frac{\pi}2\right)$$
respectively. For example, he conjectured the congruence
\begin{equation}\label{S-256}\sum_{n=0}^{p-1}\frac{6n+1}{256^n}\binom{2n}n^3\equiv (-1)^{(p-1)/2}p-p^3E_{p-3}\pmod{p^4}
\end{equation}
for any prime $p>3$. This was later confirmed by G.-S. Mao and C.-W. Wen \cite[Th. 1.2]{MW}.

In 2019 Z.-W. Sun \cite[Conj. 22]{OpenConj} conjectured that for any prime $p>3$ and positive odd integer $m$ we have
\begin{align*}&\frac{16^{m-1}}{(pm)^4\binom{m-1}{\frac{m-1}2}^3}
\bigg(\sum_{n=0}^{\frac{pm-1}2}\frac{6n+1}{256^n}\binom{2n}n^3
-(-1)^{(p-1)/2}p\sum_{r=0}^{\frac{m-1}2}\frac{6r+1}{256^r}\binom{2r}r^3\bigg)
\\\quad &\equiv(-1)^{(p-1)/2}\frac 7{24}B_{p-3}\pmod p.
\end{align*}
In this paper we confirm this in the case $m=1$. Namely, we establish the following result.

\begin{thm}\label{Thsun1} Let $p>3$ be a prime. Then
\begin{equation}\label{6n112256}
\sum_{n=0}^{(p-1)/2}\frac{6n+1}{256^n}\binom{2n}n^3\equiv p(-1)^{(p-1)/2}+(-1)^{(p-1)/2}\frac{7}{24}p^4B_{p-3}\pmod{p^5}.
\end{equation}
\end{thm}

Another similar congruence modulo $p^5$ states that
$$\sum_{n=0}^{p-1}\frac{3n+1}{16^n}\binom{2n}n^3\equiv p+\frac 76 p^4\pmod{p^5}$$
for any prime $p>3$, which was conjectured by Sun \cite{sun-scm-2011} in 2011 and
confirmed by C. Wang and D.-W. Hu \cite{WH} in 2020.

In the next section, we provide some known lemmas.
We will use the WZ method to prove Theorem \ref{Thsun1} in Section 3.

\section{Some known lemmas}

In 1862 J. Wolstenholme \cite{wolstenholme-qjpam-1862} proved the classical congruence
$$\bi{2p-1}{p-1}\eq1\pmod{p^3}$$
for any prime $p>3$. This was  refined by J.W.L. Glaisher \cite{Gl} in 1900.

\begin{lem} [Glaisher \cite{Gl}]\label{Lem-Gl} For any prime $p>3$, we have
\begin{equation}\label{Gl}\bi{2p-1}{p-1}\eq1-\f23 p^3B_{p-3}\pmod{p^4}.
\end{equation}
\end{lem}
\begin{remark} For modern references about \eqref{Gl}, the reader may consult \cite{Mc} and \cite{HT}.
\end{remark}

In 1895, F. Morley \cite{Mor} got the following fundamental congruence:
$$\binom{p-1}{(p-1)/2}\equiv(-1)^{(p-1)/2}4^{p-1}\pmod{p^3}$$
for any prime $p>3$. This was refined by L. Carlitz \cite{C} in 1953.

\begin{lem}[Carlitz\cite{C}]\label{Lemc} For each odd prime $p$, we have
$$(-1)^{(p-1)/2}\binom{p-1}{(p-1)/2}\equiv4^{p-1}+\frac{p^3}{12}B_{p-3}\pmod{p^4}.$$
\end{lem}

We also need the following result of E. Lehmer established in 1938.

\begin{lem} [E. Lehmer \cite{Leh}] \label{Lem-Le} For any prime $p>3$, we have
\begin{equation}\label{Lem-H}\sum_{k=1}^{(p-1)/2}\f1k\eq-2\sum_{k=1}^{(p-1)/2}\f1{2k-1}\eq-2 q_p(2)+ p\,q_p(2)^2\pmod{p^2},\end{equation}
where $q_p(2)$ denotes the Fermat quotient $(2^{p-1}-1)/p$.
\end{lem}

Let $a_1,a_2,\ldots,a_m$ be integers. For any integer $n\geq m$, we define the alternating multiple harmonic sum
$$
H(a_1,a_2,\ldots,a_m;n):=\sum_{\substack{1\leq k_1<k_2<\ldots<k_m\leq n}}\prod_{i=1}^m\frac{\mbox{sign}(a_i)^{k_i}}{k_i^{|a_i|}},
$$
and call $m$ and $\sum_{i=1}^m|a_i|$ its {\it depth} and {\it weight} respectively. For convenience, we simply write $H_n$ to stand for $H(1;n)$.

We need the following known results as lemmas.

\begin{lem}[\cite{H}] \label{(i)} Let $a,r\in\Z^+=\{1,2,3,\ldots\}$, For any prime $p>ar+2$, we have
\begin{align*}
H(\{a\}^r;p-1)\equiv\begin{cases}(-1)^r\frac{a(ar+1)}{2(ar+2)}p^2B_{p-ar-2}\ &\pmod{p^3}\qquad \text{if} \ ar\ \text{is odd},\\
(-1)^{r-1}\frac{a}{ar+1}p B_{p-ar-1}\ &\pmod{p^2}\qquad \text{if}\ ar\ \t{is even}.\end{cases}
\end{align*}
\end{lem}

\begin{lem}[\cite{SH1}] \label{(ii)} For any $a\in\Z^+$ and prime $p>a+2$, we have
\begin{align*}
&H\left(a;\frac{p-1}2\right)\\
&\equiv\begin{cases}-2q_p(2)+pq_p(2)^2-\frac23p^2q_p(2)^3-\frac7{12}p^2B_{p-3}\pmod{p^3} &\text{if}\ a=1,\\
-\frac{2^a-2}{a}B_{p-a} \pmod{p} &\text{if $a>1$ is odd},\\
\frac{a(2^{a+1}-1)}{2(a+1)}pB_{p-a-1} \pmod{p^2} &\text{if $a$ is even}.\end{cases}
\end{align*}
\end{lem}

\begin{lem}[\cite{H}]\label{(iii)} For any  $a,b\in\Z^+$ and any prime $p>a+b+1$, we have
$$
H(a,b;p-1)\equiv\frac{(-1)^b}{a+b}\binom{a+b}aB_{p-a-b}\pmod p.
$$
\end{lem}

\begin{lem}[\cite{RT}]\label{(iv)}
For any prime $p>3$, we have
\begin{align*}
H(1,-2;p-1)&\equiv H(-1,2;p-1)\equiv H(2,-1;p-1)\\
&\equiv H(-2,1;p-1)\equiv\frac14B_{p-3}
\pmod p
\end{align*}
and
$$
H(1,1,-1;p-1)\equiv-\frac13q_p(2)^3-\frac7{24}B_{p-3}\pmod p.
$$
\end{lem}

\begin{lem}[\cite{HHT}]\label{(v)} Let $a,b\in\Z^+$ with $a+b$ odd. For any prime $p>a+b$, we have
$$
H\left(a,b;\frac{p-1}2\right)\equiv\frac{B_{p-a-b}}{2(a+b)}\left((-1)^b\binom{a+b}a+2^{a+b}-2\right)\pmod p.
$$
\end{lem}

\begin{lem}[R. Tauraso and J. Q. Zhao \cite{TZ}]\label{TZ}  For any prime $p>3$, we have
\begin{equation}\label{TZp2}
H(1,-1;p-1)\equiv q_p(2)^2-pq_p(2)^3-\frac{13}{24}pB_{p-3}\pmod{p^2}.
\end{equation}
\end{lem}

\section{Proof of Theorem \ref{Thsun1}}

We will use the following WZ pair appeared in \cite{CXH-rama-2016} to prove Theorem \ref{Thsun1}. For $n, k\in\N=\{0,1,2,\ldots\}$, we define
$$
F(n,k)=\frac{(6n-2k+1)}{2^{8n-2k}}\frac{\binom{2n}n\binom{2n+2k}{n+k}\binom{2n-2k}{n-k}\binom{n+k}{n}}{\binom{2k}k}
$$
and
$$
G(n,k)=\frac{n^2\binom{2n}n\binom{2n+2k}{n+k}\binom{2n-2k}{n-k}\binom{n+k}{n}}{2^{8n-2k-4}(2n+2k-1)\binom{2k}k}.
$$
Clearly $F(n,k)=G(n,k)=0$ if $n<k$. It is easy to check that
\begin{equation}\label{FG}
F(n,k-1)-F(n,k)=G(n+1,k)-G(n,k)
\end{equation}
for all $n\in\N$ and $k\in\Z^+$.

Summing (\ref{FG}) over $n\in\{0,\ldots,(p-1)/2\}$ we get
$$
\sum_{n=0}^{(p-1)/2}F(n,k-1)-\sum_{n=0}^{(p-1)/2}F(n,k)=G\left(\frac{p+1}2,k\right)-G(0,k)=G\left(\frac{p+1}2,k\right).
$$
Furthermore, summing both side of the above identity over $k\in\{1,\ldots,(p-1)/2\}$, we obtain
\begin{align}\label{wz2}
\sum_{n=0}^{(p-1)/2}F(n,0)=F\left(\frac{p-1}2,\frac{p-1}2\right)+\sum_{k=1}^{(p-1)/2}G\left(\frac{p+1}2,k\right).
\end{align}
\begin{lem}\label{Fp12p12} Let $p>3$ be a prime. Then
\begin{align*}
&F\left(\frac{p-1}2,\frac{p-1}2\right)\\
&\equiv (-1)^{\frac{p-1}2}p\left(1-pq_p(2)+p^2q_p(2)^2-p^3q_p(2)^3-\frac7{12}p^3B_{p-3}\right)\pmod{p^5}.
\end{align*}
\end{lem}
\begin{proof} By the definition of $F(n,k)$, we have
\begin{align*}
F\left(\frac{p-1}2,\frac{p-1}2\right)&=\frac{2p-1}{2^{3p-3}}\binom{2p-2}{p-1}\binom{p-1}{(p-1)/2}=\frac{p\binom{2p-1}{p-1}\binom{p-1}{(p-1)/2}}{2^{3p-3}}.
\end{align*}
This, together with Lemma \ref{Lem-Gl}, Lemma \ref{Lemc} and the equality $2^{p-1}=1+pq_p(2)$, yields that
\begin{align*}
&F\left(\frac{p-1}2,\frac{p-1}2\right)\equiv\frac{p(1-\frac23p^3B_{p-3})(-1)^{(p-1)/2}(4^{p-1}+\frac1{12}p^3B_{p-3})}{(1+pq_p(2))^3}\\
&\equiv(-1)^{\frac{p-1}2}p\left(1-pq_p(2)+p^2q_p(2)^2-p^3q_p(2)^3-\frac7{12}p^3B_{p-3}\right)\pmod{p^5}.
\end{align*}
This concludes the proof.
\end{proof}

\begin{lem}\label{tuo} For any prime $p>3$, we have
\begin{align*}
\sum_{k=1}^{(p-1)/2}\frac{(p/2-k)}{(p+1-2k)(p+2k)}&\equiv\frac12q_p(2)-\frac{p}4q_p(2)^2-2pq_p(2)+\frac16p^2q_p(2)^3\\
&+4p^2q_p(2)+p^2q_p(2)^2+\frac{7}{48}p^2B_{p-3}\pmod{p^3}.
\end{align*}
\end{lem}

\begin{proof} In view of Lemma \ref{Lem-Le},
\begin{align}\label{k2k-1}
&\sum_{k=1}^{(p-1)/2}\frac1{k(2k-1)}\notag\\
&=2\sum_{k=1}^{(p-1)/2}\frac1{2k-1}-H_{(p-1)/2}\equiv4q_p(2)-2pq_p(2)^2\pmod{p^2}
\end{align}
and
\begin{align}\label{k2k-12}
\sum_{k=1}^{(p-1)/2}\frac1{k(2k-1)^2}&=H_{(p-1)/2}-2\sum_{k=1}^{(p-1)/2}\frac1{2k-1}+2\sum_{k=1}^{(p-1)/2}\frac1{(2k-1)^2}\notag\\
&\equiv-4q_p(2)+\frac12H(2;(p-1)/2)\notag\\
&\equiv-4q_p(2)\pmod p.
\end{align}
It is easy to see that
\begin{align*}
&\sum_{k=1}^{(p-1)/2}\frac{(p/2-k)}{(p+1-2k)(p+2k)}\\
&\equiv-\frac14H_{(p-1)/2}-\frac{p}2\sum_{k=1}^{(p-1)/2}\frac1{k(2k-1)}-p^2\sum_{k=1}^{(p-1)/2}\frac1{k(2k-1)^2}.
\end{align*}
Then we immediately obtain the desired result by Lemma \ref{Lem-Le}, (\ref{k2k-1}) and (\ref{k2k-12}).
\end{proof}

\begin{lem}\label{Lemhk} For any prime $p>3$, we have
\begin{align*}
&\sum_{k=1}^{(p-1)/2}\frac{(p/2-k)H_k}{(p+1-2k)(p+2k)}\\
&\equiv2q_p(2)-q_p(2)^2-6pq_p(2)+2pq_p(2)^2+pq_p(2)^3+\frac7{12}pB_{p-3}\pmod{p^2}.
\end{align*}
\end{lem}
\begin{proof} By Lemmas \ref{(i)} and \ref{(ii)},  and (\ref{TZp2}), we have
\begin{align}\label{2kk}
\sum_{k=1}^{(p-1)/2}\frac{H_{2k}}k&=\sum_{k=1}^{p-1}\frac{(1+(-1)^k)H_{k}}k\notag\\
&=H(1,1;p-1)+H(1,-1;p-1)+\frac12H(2;(p-1)/2)\notag\\
&\equiv q_p(2)^2-pq_p(2)^3+\frac{7}{24}pB_{p-3}\pmod{p^2}.
\end{align}
Noting $2H(1,1;n)=H_n^2-H(2;n)$, we get
\begin{align}\label{kk}
\sum_{k=1}^{(p-1)/2}\frac{H_{k}}k&=H(1,1;(p-1)/2)+H(2;(p-1)/2)\notag\\&=\frac12H_{(p-1)/2}^2+\frac12H(2;(p-1)/2)\notag\\
&\equiv2q_p(2)^2-2pq_p(2)^3+\frac76pB_{p-3}\pmod{p^2}.
\end{align}
It is easy to see that
\begin{align*}
&H_{(p+1)/2-k}\\
&\equiv\frac2{p+1-2k}+2pH(2;2k)-\frac{p}2H(2;k)+H_{(p-1)/2}+2H_{2k}-H_k\pmod{p^2}.
\end{align*}
This, together with (\ref{2kk}), (\ref{kk}) and \cite[(2.2), (2.3)]{mw-ijnt-2019}, yields that
\begin{align}\label{p+12-kp2}
\sum_{k=1}^{(p-1)/2}\frac{H_{(p+1)/2-k}}k&\equiv-8q_p(2)+4q_p(2)^2+4pq_p(2)^2+8pq_p(2)\notag\\
&-4pq_p(2)^3-\frac73pB_{p-3}\pmod{p^2}
\end{align}
and
\begin{align}\label{p+12-k}
\sum_{k=1}^{(p-1)/2}\frac{H_{(p+1)/2-k}}{k(2k-1)}&\equiv\sum_{k=1}^{(p-1)/2}\frac{H_k}{k(2k-1)}\notag\\
&\equiv-\sum_{k=1}^{(p-1)/2}\frac{H_{(p+1)/2-k}}k-\sum_{k=1}^{(p-1)/2}\frac{H_{k}}k\notag\\
&\equiv8q_p(2)-6q_p(2)^2\pmod p.
\end{align}
Since
\begin{align*}
&\sum_{k=1}^{(p-1)/2}\frac{(p/2-k)H_k}{(p+1-2k)(p+2k)}\\
&\equiv-\frac14\sum_{k=1}^{(p-1)/2}\frac{H_{(p+1)/2-k}}k
-\frac{p}2\sum_{k=1}^{(p-1)/2}\frac{H_{(p+1)/2-k}}{k(2k-1)}\pmod{p^2},
\end{align*}
we immediately get the desired result with the aids of (\ref{p+12-kp2}) and (\ref{p+12-k}).
\end{proof}

\begin{lem}\label{Lemh2k} Let $p>3$ be a prime. Then 
\begin{align*}
&\sum_{k=1}^{(p-1)/2}\frac{(p/2-k)H_{2k}}{(p+1-2k)(p+2k)}\\
&\equiv q_p(2)-\frac14q_p(2)^2-3pq_p(2)+\frac{p}2q_p(2)^2+\frac{p}4q_p(2)^3+\frac{13}{32}pB_{p-3}\pmod{p^2}.
\end{align*}
\end{lem}
\begin{proof}
It is easy to check that
$$H_{p+1-2k}\equiv pH(2;2k-2)+H_{2k-2}\pmod{p^2}.$$
So
\begin{align*}
&\sum_{k=1}^{(p-1)/2}\frac{H_{p+1-2k}}k\equiv p\sum_{k=1}^{(p-1)/2}\frac{H(2;2k-2)}k+\sum_{k=1}^{(p-1)/2}\frac{H_{2k-2}}k\\
=&p\left(\sum_{k=1}^{(p-1)/2}\frac{H(2;2k)}k-\sum_{k=1}^{(p-1)/2}\frac1{k(2k-1)^2}-\frac14H(3;(p-1)/2)\right)\\
&+\left(\sum_{k=1}^{(p-1)/2}\frac{H_{2k}}k-\sum_{k=1}^{(p-1)/2}\frac1{k(2k-1)}-\frac12H(2;(p-1)/2)\right).
\end{align*}
Combing this with Lemma \ref{(ii)}, (\ref{2kk}), (\ref{k2k-1}), (\ref{k2k-12}) and \cite[(2.3)]{mw-ijnt-2019}, we get
\begin{align}\label{p+12-2kp2}
\sum_{k=1}^{(p-1)/2}\frac{H_{p+1-2k}}k&\equiv q_p(2)^2-4q_p(2)+4pq_p(2)+2pq_p(2)^2\notag\\
&-pq_p(2)^3-\frac{13}8pB_{p-3}\pmod{p^2}
\end{align}
and
\begin{align}\label{p+12-2k}
\sum_{k=1}^{(p-1)/2}\frac{H_{p+1-2k}}{k(2k-1)}\equiv\sum_{k=1}^{(p-1)/2}\frac{H_{2k-2}}{k(2k-1)}\equiv4q_p(2)-2q_p(2)^2\pmod p.
\end{align}
Since
\begin{align*}
\sum_{k=1}^{(p-1)/2}\frac{(p/2-k)H_{2k}}{(p+1-2k)(p+2k)}
=-\frac14\sum_{k=1}^{(p-1)/2}\frac{H_{p+1-2k}}k-\frac{p}2\sum_{k=1}^{(p-1)/2}\frac{H_{p+1-2k}}{k(2k-1)},
\end{align*}
we obtain the desired result by using (\ref{p+12-2kp2}) and (\ref{p+12-2k}).
\end{proof}

\begin{lem}\label{Lemhk2} For any prime $p>3$, we have
$$
\sum_{k=1}^{(p-1)/2}\frac{(p/2-k)H_{k}^2}{(p+1-2k)(p+2k)}\equiv4q_p(2)-6q_p(2)^2+2q_p(2)^3+\frac18B_{p-3}\pmod p.
$$
\end{lem}
\begin{proof} It is easy to verify that
\begin{align*}
&\sum_{k=1}^{(p-1)/2}\frac{(p/2-k)H_{k}^2}{(p+1-2k)(p+2k)}\equiv\frac12\sum_{k=1}^{(p-1)/2}\frac{H_k^2}{2k-1}=\frac12\sum_{k=0}^{(p-3)/2}\frac{H_{k+1}^2}{2k+1}\\
&=\frac12\sum_{k=0}^{(p-3)/2}\frac{H_k^2}{2k+1}+\frac12\sum_{k=1}^{(p-1)/2}\frac1{(2k-1)k^2}+\sum_{k=1}^{(p-1)/2}\frac{H_{k-1}}{k(2k-1)}.
\end{align*}
Observe that
\begin{align}\label{2k-2k2}
&\frac12\sum_{k=1}^{(p-1)/2}\frac1{(2k-1)k^2}\notag\\
&=2\sum_{k=1}^{(p-1)/2}\frac1{2k-1}-H_{(p-1)/2}-\frac12H(2;(p-1)/2)
\end{align}
and
\begin{align*}
&\sum_{k=1}^{(p-1)/2}\frac{H_{k-1}}{k(2k-1)}=2\sum_{k=1}^{(p-1)/2}\frac{H_{k-1}}{2k-1}-\sum_{k=1}^{(p-1)/2}\frac{H_{k-1}}{k}\\
&=2\sum_{k=1}^{(p-1)/2}\frac{H_{k}}{2k-1}-2\sum_{k=1}^{(p-1)/2}\frac{1}{k(2k-1)}-\sum_{k=1}^{(p-1)/2}\frac{H_{k-1}}{k}\\
&\equiv-\sum_{k=1}^{(p-1)/2}\frac{H_{(p+1)/2-k}}{k}-2\sum_{k=1}^{(p-1)/2}\frac{1}{k(2k-1)}-\sum_{k=1}^{(p-1)/2}\frac{H_{k-1}}{k}.
\end{align*}
This, together with (\ref{Lem-H}), Lemma \ref{(ii)}, (\ref{k2k-1}), (\ref{p+12-kp2}) and \cite[(1.1)]{mw-ijnt-2019}, yields the desired result.
\end{proof}

\begin{lem}\label{Lemhkh2k} Let $p>3$ be a prime. Then
$$
\sum_{k=1}^{(p-1)/2}\frac{(p/2-k)H_{k}H_{2k}}{(p+1-2k)(p+2k)}\equiv2q_p(2)-\frac52q_p(2)^2+\frac12q_p(2)^3+\frac5{16}B_{p-3}\pmod p.
$$
\end{lem}
\begin{proof} By \cite[Lemma 2.4, (3.12)]{mw-ijnt-2019}, (\ref{p+12-2k}), and Lemmas \ref{(i)}, \ref{(ii)} and \ref{(iv)}, we have
\begin{align}\label{h2k2k-2}
&\sum_{k=1}^{(p-1)/2}\frac{H_{2k}H_{2k-2}}{k}=\sum_{k=1}^{(p-1)/2}\frac{H_{2k}^2}{k}-\sum_{k=1}^{(p-1)/2}\frac{H_{2k}}{k(2k-1)}-\frac12\sum_{k=1}^{(p-1)/2}\frac{H_{2k}}{k^2}\notag\\
&\equiv-4q_p(2)+2q_p(2)^2-\frac23q_p(2)^3-\frac1{12}B_{p-3}\pmod p.
\end{align}
In view of \cite[Lemma 3.2]{mao-ijnt-2017}, \cite[Theorem 1.3]{mw-ijnt-2019} and (\ref{p+12-kp2}),
 we have
\begin{align}\label{hk2k-2}
&\sum_{k=1}^{(p-1)/2}\frac{H_{k}H_{2k-2}}{k}=\sum_{k=1}^{(p-1)/2}\frac{H_{2k}H_{k}}{k}-\sum_{k=1}^{(p-1)/2}\frac{H_{k}}{k(2k-1)}-\frac12\sum_{k=1}^{(p-1)/2}\frac{H_{k}}{k^2}\notag\\
&\equiv-8q_p(2)+6q_p(2)^2-\frac43q_p(2)^3+\frac{13}{12}B_{p-3}\pmod p.
\end{align}
It is easy to see that
\begin{align*}
&\sum_{k=1}^{(p-1)/2}\frac{(p/2-k)H_{k}H_{2k}}{(p+1-2k)(p+2k)}\\
&\equiv\frac12\sum_{k=1}^{(p-1)/2}\frac{H_{2k-2}}{k(2k-1)}-\frac14H_{(p-1)/2}\sum_{k=1}^{(p-1)/2}\frac{H_{2k-2}}{k}\\
&-\frac12\sum_{k=1}^{(p-1)/2}\frac{H_{2k}H_{2k-2}}{k}+\frac14\sum_{k=1}^{(p-1)/2}\frac{H_{k}H_{2k-2}}{k}.
\end{align*}
 Combining this with (\ref{h2k2k-2}), (\ref{hk2k-2}), (\ref{p+12-2k}), (\ref{2kk}), (\ref{k2k-1}) and Lemma \ref{(ii)}, we immediately get the desired result.
\end{proof}

\begin{lem}\label{Lemh2k2} For any prime $p>3$, we have
$$
\sum_{k=1}^{(p-1)/2}\frac{(p/2-k)H_{2k}^2}{(p+1-2k)(p+2k)}\equiv q_p(2)-q_p(2)^2+\frac16q_p(2)^3+\frac13B_{p-3}\pmod p.
$$
\end{lem}
\begin{proof} Replacing $k$ by $(p+1)/2-j$ in (\ref{k2k-12}), we have
$$
\sum_{j=1}^{(p-1)/2}\frac1{(2j-1)j^2}\equiv8q_p(2)\pmod p,
$$
and in view of \cite[Lemma 2.4, (3.12)]{mw-ijnt-2019} and (ii), we can deduce that
$$
\sum_{k=1}^{(p-1)/2}\frac{H_{2k-2}}{k^2}\equiv-8q_p(2)+\frac52B_{p-3}\pmod p.
$$
This, together with (\ref{p+12-2k}) and (\ref{h2k2k-2}), yields that
\begin{align*}
&\sum_{k=1}^{(p-1)/2}\frac{(p/2-k)H_{2k}^2}{(p+1-2k)(p+2k)}\\
&\equiv-\frac{1}4\left(\sum_{k=1}^{(p-1)/2}\frac{H_{2k}H_{2k-2}}{k}-\sum_{k=1}^{(p-1)/2}\frac{H_{2k-2}}{k(2k-1)}-\frac12\sum_{k=1}^{(p-1)/2}\frac{H_{2k-2}}{k^2}\right)\\
&\equiv q_p(2)-q_p(2)^2+\frac16q_p(2)^3+\frac13B_{p-3}\pmod p.
\end{align*}
This ends the proof.
\end{proof}
\begin{lem}\label{Lemh2kk} Let $p>3$ be a prime. Then
$$\sum_{k=1}^{(p-1)/2}\frac{(p/2-k)H(2;2k)}{(p+1-2k)(p+2k)}\equiv q_p(2)-\frac3{16}B_{p-3}\pmod p,$$
$$\sum_{k=1}^{(p-1)/2}\frac{(p/2-k)H(2;k)}{(p+1-2k)(p+2k)}\equiv 4q_p(2)-\frac7{8}B_{p-3}\pmod p.$$
\end{lem}
\begin{proof} In view of \cite[(2.3)]{mw-ijnt-2019}, (\ref{k2k-12}) and Lemma \ref{(ii)}, we have
\begin{align*}
&\sum_{k=1}^{(p-1)/2}\frac{(p/2-k)H(2;2k)}{(p+1-2k)(p+2k)}\\
&\equiv\frac14\left(\sum_{k=1}^{(p-1)/2}\frac{H(2;2k)}{k}-\sum_{k=1}^{(p-1)/2}\frac{1}{k(2k-1)^2}-\frac14\sum_{k=1}^{(p-1)/2}\frac{1}{k^3}\right)\\
&\equiv q_p(2)-\frac{3}{16}B_{p-3}\pmod p.
\end{align*}
It is easy to see that
\begin{align*}
\sum_{k=1}^{(p-1)/2}\frac{(p/2-k)H(2;k)}{(p+1-2k)(p+2k)}&\equiv\frac12\sum_{k=1}^{(p-1)/2}\frac{H(2;(p+1)/2-k)}{p-2k}\\
&\equiv\sum_{k=1}^{(p-1)/2}\frac{H(2;2k-2)}k-\frac14\sum_{k=1}^{(p-1)/2}\frac{H(2;k-1)}k.
\end{align*}
Observe that
\begin{align*}
\sum_{k=1}^{(p-1)/2}\frac{H(2;2k-2)}k&=\sum_{k=1}^{(p-1)/2}\frac{H(2;2k)}k-\sum_{k=1}^{(p-1)/2}\frac{1}{k(2k-1)^2}-\frac14\sum_{k=1}^{(p-1)/2}\frac{1}{k^3}\\
&\equiv4q_p(2)-\frac34B_{p-3}\pmod p
\end{align*}
and
\begin{align*}
\sum_{k=1}^{(p-1)/2}\frac{H(2;k-1)}k=\sum_{k=1}^{(p-1)/2}\frac{H(2;k)}k-\sum_{k=1}^{(p-1)/2}\frac{1}{k^3}\equiv\frac12B_{p-3}\pmod p.
\end{align*}
So
$$\sum_{k=1}^{(p-1)/2}\frac{(p/2-k)H(2;k)}{(p+1-2k)(p+2k)}\equiv 4q_p(2)-\frac7{8}B_{p-3}\pmod p.$$
Therefore the proof of Lemma \ref{Lemh2kk} is complete.
\end{proof}
\begin{lem} \label{G13} For any primes $p>3$, we have
\begin{align*}
&\sum_{k=1}^{(p-1)/2}G\left(\frac{p+1}2,k\right)\\
&\equiv (-1)^{\frac{p-1}2}p^2\left(q_p(2)-pq_p(2)^2+p^2q_p(2)^3+\frac7{8}p^3B_{p-3}\right)\pmod{p^5}.
\end{align*}
\end{lem}
\begin{proof} For any complex number $a$, let $(a)_0=1$ and $(a)_n=a(a+1)\ldots(a+n-1)$
 for $n\in\Z^+$.
By the definition of $G(n,k)$, we have
\begin{align}\label{Gpk1}
G(n,k)&=\frac{n^2\binom{2n}{n}\binom{2n+2k}{n+k}\binom{2n-2k}{n-k}\binom{n+k}{n}}{2^{8n-4-2k}(2n+2k-1)\binom{2k}k}=\frac{n^2\binom{2n}{n}\left(\frac12\right)_{n+k}\left(\frac12\right)_{n-k}\binom{n}{k}}{2^{4n-4-2k}n!^2(2n+2k-1)\binom{2k}k}\notag\\
&=\frac{n^2\binom{2n}{n}\left(\frac12\right)_{n}\left(\frac12\right)_{n-1}\left(\frac12+n\right)_{k}\binom{n}{k}}{2^{4n-4-2k}n!^2\left(\frac12+n-k\right)_{k-1}(2n+2k-1)\binom{2k}k}\notag\\
&=\frac{n\binom{2n}{n}^2\binom{2n-2}{n-1}\left(\frac12+n\right)_{k}\binom{n}{k}}{2^{8n-6-2k}n!^2\left(\frac12+n-k\right)_{k-1}(2n+2k-1)\binom{2k}k},
\end{align}
where we have used the equalities
$$
\frac{\left(\frac12\right)_{n+k}}{(n+k)!}=\frac{\binom{2n+2k}{n+k}}{4^{n+k}},
$$
$$
\left(\frac12\right)_{n+k}=\left(\frac12\right)_{n}\left(\frac12+n\right)_{k}
$$
and
$$
\left(\frac12\right)_{n-k}\left(\frac12+n-k\right)_{k-1}=\left(\frac12\right)_{n-1}\ \ (1\le k\le n).
$$

It is easy to check that
\begin{align*}
\frac{\left(\frac{p}2+1\right)_k}{\left(\frac{p}2-k\right)_k}&\equiv\frac{k!\left(1+\frac{p}2H_k+\frac{p^2}4\sum_{1\leq i<j\leq k}\frac1{ij}\right)}{(-1)^kk!(\left(1-\frac{p}2H_k+\frac{p^2}4\sum_{1\leq i<j\leq k}\frac1{ij}\right)}\\
&\equiv(-1)^k\left(1+pH_k+\frac{p^2}2H_k^2\right)\pmod{p^3}.
\end{align*}
In view of \cite[(4.4)]{sun-ijm-2015}, we have the following congruence modulo $p^3$
\begin{align*}
&\frac{\binom{(p-1)/2}k(-4)^k}{\binom{2k}k}\equiv1-p\sum_{j=1}^k\frac1{2j-1}+\frac{p^2}2\left(\left(\sum_{j=1}^k\frac1{2j-1}\right)^2-\sum_{j=1}^k\frac1{(2j-1)^2}\right)\\
&=1-p\left(H_{2k}-\frac12H_{k}\right)+\frac{p^2}2\left(\left(H_{2k}-\frac12H_k\right)^2-H(2;2k)+\frac14H(2;k)\right).
\end{align*}
By (\ref{Gpk1}), we have the following congruence modulo $p^5$
\begin{align*}
&\sum_{k=1}^{(p-1)/2}G\left(\frac{p+1}2,k\right)\equiv\frac{(p+1)^2\binom{p+1}{(p+1)/2}^2\binom{p-1}{(p-1)/2}}{2^{4p-1}}\sum_{k=1}^{(p-1)/2}\frac{(p/2-k)}{(p+1-2k)(p+2k)}\\
&\cdot \left(1+\frac{3p}2H_k-pH_{2k}+\frac{9p^2}8H_k^2-\frac{3p^2}2H_kH_{2k}+\frac{p^2}2H_{2k}^2-\frac{p^2}2\left(H(2;2k)-\frac{H(2;k)}4\right)\right).
\end{align*}
In view of Lemmas \ref{tuo}--\ref{Lemh2kk} and Lemma \ref{Lemc}, we have the following congruence modulo $p^5$
\begin{align*}
&\sum_{k=1}^{(p-1)/2}G\left(\frac{p+1}2,k\right)\\
&\equiv\frac{(p+1)^2\binom{p+1}{\frac{p+1}2}^2\binom{p-1}{\frac{p-1}2}}{2^{4p-1}}\left(\frac12q_p(2)-\frac32pq_p(2)^2+3p^2q_p(2)^3+\frac7{16}p^2B_{p-3}\right)\\
&\equiv\frac{2p^2\binom{p-1}{(p-1)/2}^3}{2^{4p-4}}\left(\frac12q_p(2)-\frac32pq_p(2)^2+3p^2q_p(2)^3+\frac7{16}p^2B_{p-3}\right)\\
&\equiv2(-1)^{(p-1)/2}p^24^{p-1}\left(\frac12q_p(2)-\frac32pq_p(2)^2+3p^2q_p(2)^3+\frac7{16}p^2B_{p-3}\right).
\end{align*}
Then we obtain the desired result by noting that $4^{p-1}=1+2pq_p(2)+p^2q_p(2)^2$.
\end{proof}

\medskip
\noindent{\it Proof of Theorem \ref{Thsun1}}. Substituting Lemmas \ref{Fp12p12} and \ref{G13} into (\ref{wz2}),  we immediately get that
$$
\sum_{n=0}^{(p-1)/2}F(n,0)\equiv p(-1)^{(p-1)/2}+(-1)^{(p-1)/2}\frac7{24}p^4B_{p-3}\pmod{p^5},
$$
which is equivalent to our desired result. \qed

\medskip

 \noindent{\bf Acknowledgments.}
The first author is funded by the National Natural Science Foundation of China (grant no. 12001288), and the second author is supported by the National Natural Science Foundation of China (grant no. 11971222).

\end{document}